\documentclass[12pt]{article}
\usepackage{amssymb}
\usepackage{amscd}
\usepackage{amsfonts}
\usepackage{times}
\usepackage{mathptmx}
\usepackage{amsmath}
\usepackage[usenames]{color}
\usepackage{mathrsfs}
\usepackage{amsfonts}
\usepackage{amssymb,amsmath}
\usepackage{CJK}
\usepackage{cite}
\usepackage{cases}
\usepackage{amsthm}
\usepackage{extarrows}
\usepackage{indentfirst,latexsym,bm}
\usepackage{amsthm}
\usepackage{fancybox}
\def\cl{\centerline}
\def\la{\lambda}

\def\b{\beta}
\def\pa{\partial}
\def\vs{\vspace*}

\def\C{\mathbb{C}}

\def\Vir{\hbox{Vir}}

\def\vs{\vspace*}

\numberwithin{equation}{section}
\newtheorem{theo}{Theorem}[section]
\newtheorem{defi}[theo]{Definition}

\newtheorem{lemm}[theo]{Lemma}
\newtheorem{prop}[theo]{Proposition}

\begin{document}
\begin{center}
{\bf\large Extensions of modules over a class of  Lie conformal algebras $\mathcal{W}(b)$}
\footnote {$^{\,*}$Corresponding author: lmyuan@hit.edu.cn (Lamei Yuan)}
\end{center}

\cl{Kaijing Ling, Lamei Yuan $^{\,*}$}

\cl{\small Department of Mathematics, Harbin Institute of Technology, Harbin
150001, China}

\cl{\small\footnotesize E-mails:
kjling\_edu@126.com, lmyuan@hit.edu.cn}
\vs{8pt}

{\small
\parskip .005 truein
\baselineskip 3pt \lineskip 3pt
\noindent{{\bf Abstract:} Let $\mathcal{W}(b)$ be a class of free Lie conformal algebras of rank $2$ with $\C[\pa]$-basis $\{L,H\}$ and relations
\begin{eqnarray*}
[L_\la L]=(\pa+2\la)L,\ \ [L_\la H]=\big(\pa+(1-b)\la\big)H, \ \ [H_\la H]=0,
\end{eqnarray*}
where $b$ is a nonzero complex number. In this paper, we classify extensions between two finite irreducible conformal modules over the Lie conformal algebras $\mathcal{W}(b)$.
 \vs{5pt}


\vs{5pt}

\noindent{\bf MR(2000) Subject Classification:}~ 17B65, 17B68

\parskip .001 truein\baselineskip 6pt \lineskip 6pt

\section{Introduction}
\vs{8pt}

Conformal modules of Lie conformal algebras were introduced in \cite{DK}.
Cheng and Kac studied finite conformal modules over the Virasoro, the Neveu-Schwarz and the current conformal algebras in \cite{CK}. In general, conformal modules are not completely reducible. Thus it is necessary to study the extension problem. In \cite{CKW1,CKW2}, Cheng, Kac and Wakimoto classified extensions between finite conformal modules over the Virasoro, the current, the Neveu-Schwarz and the semi-direct sum of the Virasoro and the current conformal algebras. In \cite{L}, Ngau Lam solved the same problem for the supercurrent conformal algebras by using the techniques developed in \cite{CKW1}.

In this paper, we aim to study extensions between conformal modules over a class of Lie conformal algebras $\mathcal{W}(b)$, which was introduced in \cite{SY} as the Lie conformal algebras associated with the Lie algebras $\mathrm{W}(a,b)$ with $a,b\in\C$. By definition, the Lie conformal algebra $\mathcal{W}(b)=\C[\pa] L\bigoplus \C[\pa] H$ with the $\la$-bracket
\begin{eqnarray}\label{la-brac}
[L_\la L]=(\pa+2\la)L,\ \ [L_\la H]=\big(\pa+(1-b)\la\big)H, \ \ [H_\la H]=0,
\end{eqnarray} where $b\in\C$.
Finite irreducible conformal modules over $\mathcal{W}(b)$ were classified in \cite{WY}. It turns out that any finite irreducible conformal $\mathcal{W}(b)$-module is free of rank one and of the form $M(\alpha,\beta,\Delta)=\C[\pa]v_{\Delta}$ with $(\Delta,\beta)\neq (0,0)$ if $b=0$ or $M(\alpha,\Delta)=\C[\pa]v_{\Delta}$ with $\Delta\neq0$ if $b\neq 0$ (see Proposition \ref{p1}).

For each $b$, the Lie conformal algebra $\mathcal{W}(b)$ is a semidirect product of the
Virasoro conformal algebra $\Vir=\C[\partial]L$ with $[L_\lambda L]=(\partial+2\lambda)L$ and the conformal $\Vir$-module $\C[\pa]H$ on which the action of $L$ is given by $L_\la H=\big(\pa+(1-b)\la\big)H$. Two special cases of $\mathcal{W}(b)$ should be pointed out.
One is $\mathcal{W}(0)$, which is called Heisenberg-Virasoro conformal algebra in \cite{SY}. Its cohomology was studied in \cite{YW}. The other one is
$\mathcal{W}(-1)$, which is exactly the $W(2,2)$ Lie conformal algebra. Its conformal derivations and central extensions were studied in \cite{YW1}.

 In this paper, we will consider extensions of conformal $\mathcal{W}(b)$-modules with $b\neq0$, since this problem for $\mathcal{W}(0)$-modules have been solved in \cite{LY}. The rest paper is organized as follows. In Section 2, we recall the notions of conformal modules and their extensions, and some known results. In Section 3, we consider extensions of conformal $\mathcal{W}(b)$-modules of the following three types:
\begin{eqnarray}
&&0\longrightarrow \C{c_\gamma}\longrightarrow E \longrightarrow M(\alpha,\Delta) \longrightarrow 0,\label{type11}\\
&&0\longrightarrow M(\alpha,\Delta)\longrightarrow E \longrightarrow \C c_\gamma \longrightarrow 0,\label{type22}\\
&&0\longrightarrow M(\bar\alpha,\bar\Delta)\longrightarrow E \longrightarrow M(\alpha,\Delta) \longrightarrow 0.\label{type33}
\end{eqnarray}
where $M(\alpha,\Delta)$ and $M(\bar\alpha,\bar\Delta)$ are irreducible $\mathcal{W}(b)$-modules, and
 $\C{c_\gamma}$ is the 1-dimensional $\mathcal{W}(b)$-module with $L_\lambda c_\gamma=H_\lambda c_\gamma=0$ and $\partial c_\gamma=\gamma c_\gamma$.
During the whole process, we use some results of the Virasoro conformal algebra from \cite{CKW1}.
The main results of our paper are summarized in
Theorems \ref{theo1}, \ref{theo2} and \ref{theo3}.

\section{Preliminaries}
\vs{8pt}
In this section, we recall some definitions, notations, and related results for later use. For more details, the reader is referred to \cite{CK,CKW1,WY}.

\begin{defi}\label{def1} \rm A {\it conformal module} $M$ over a Lie conformal algebra $R$
is a $\mathbb{C}[\partial]$-module endowed with a $\C$-linear map
$R\rightarrow {{\rm End}_{\mathbb{C}}M \bigotimes_{\mathbb{C}}\mathbb{C}[\lambda]}$, $a\mapsto a^{M}_\lambda $, satisfying the following conditions for all $a,b\in R$:
\begin{eqnarray*}
&&[a^M_\lambda,b^M_\mu]=a^M_\lambda b^M_\mu -b^M_\mu a^M_\lambda =[a_\lambda b]^M_{\lambda+\mu},\\
&&(\partial a)^M_\lambda =[\partial,a^M_\lambda]=-\lambda a^M_\lambda.
\end{eqnarray*}
A module $M$ over a conformal algebra $R$ is called {\it finite},
if $M$ is finitely generated over
$\mathbb{C}[\partial]$. A module $M$ over a conformal algebra $R$ is called {\it irreducible} if there is no nontrivial invariant subspace.
\end{defi}

The vector space $\C$ is viewed as a trivial module with trivial actions of $\pa$ and $R$. For a fixed nonzero complex constant $\alpha$, there
is a natural $\C[\pa]$-module structure on $\C c_\alpha$, such that $\C c_\alpha=\C$ and $\pa c_\alpha=\alpha c_\alpha$. Then
$\C c_\alpha$ becomes an $R$-module with trivial action of $R$.

Let $V$ and $W$ be two modules over a Lie conformal algebra (or a Lie algebra) $R$. An {\it extension} of $W$ by $V$ is an exact sequence of $R$-modules of the form
\begin{eqnarray}\label{Em}
0\longrightarrow V\xlongrightarrow{i} E \xlongrightarrow{p} W \longrightarrow 0.
\end{eqnarray}
Two extensions $0\longrightarrow V\xlongrightarrow{i} E \xlongrightarrow{p} W \longrightarrow 0$ and $0\longrightarrow V\xlongrightarrow{i'} E' \xlongrightarrow{p'} W \longrightarrow 0$ are said to be {\it equivalent} if there exists a commutative diagram of the form
\begin{equation*}
\begin{CD}
0@>>> V @>i>{\rm }>  E @>p>> W@>>> 0\\
@. @V{1_V}VV @V\psi VV @V{1_W}VV\\
0@>>> V @>i'>{\rm }> E' @>p'>>
W @>>> 0,
\end{CD}
\end{equation*}
where $1_V: V\rightarrow V$ and $1_W: W\rightarrow W$ are the respective identity maps and $\psi: E\rightarrow E'$ is a homomorphism of modules.

The direct sum of modules $V\oplus W$ obviously gives rise to an extension. Extensions equivalent to it are called {\it trivial extensions}. In general, an extension can be thought of as the direct sum of vector spaces $E=V\oplus W$, where $V$ is a submodule of $E$, while for $w$ in $W$ we have:
\begin{equation*}
a\cdot w=aw+\phi_a(w),\ \ a\in R,
\end{equation*}
where $\phi_a:W\rightarrow V$ is a linear map satisfying the cocycle condition: $$\phi_{[a,b]}(w)=\phi_a(bw)+a\phi_b(w)-\phi_b(aw)-b\phi_a(w),\ \, b\in R.$$ The set of these cocycles form a vector space over $\C$. Cocycles equivalent to the trivial extension are called {\it coboundaries}. They form a subspace and the quotient space by it is denoted by $\textrm{Ext}(W, V).$

For the Virasoro conformal algebra $\Vir$, it was shown in
\cite{CK} that all free nontrivial $\Vir$-modules of rank one
over $\mathbb{C}[\partial]$ are the following ones $(\Delta,
\alpha\in \mathbb{C})$:
\begin{eqnarray}
M(\alpha,\Delta)=\mathbb{C}[\partial]v_\Delta,\ \ L_\lambda
v=(\partial+\alpha+\Delta \lambda)v_\Delta.
\end{eqnarray}
The module $M(\alpha,\Delta)$ is irreducible if and only if
$\Delta\neq 0$. The module $M(\alpha,0)$ contains a unique
nontrivial submodule $(\partial +\alpha)M(\alpha,0)$ isomorphic to
$M(\alpha,1)$.  Moreover, the modules $M(\alpha,\Delta)$ with
$\Delta\neq 0$ exhaust all finite irreducible non-1-dimensional
$\Vir$-modules.

In \cite{CKW1}, extensions over the Virasoro conformal modules of the following types have been classified:
\begin{eqnarray}
&&0\longrightarrow \C{c_\gamma}\longrightarrow E \longrightarrow M(\alpha,\Delta) \longrightarrow 0,\label{type1}\\
&&0\longrightarrow M(\alpha,\Delta)\longrightarrow E \longrightarrow \C c_\gamma \longrightarrow 0,\label{type2}\\
&&0\longrightarrow M(\bar\alpha,\bar\Delta)\longrightarrow E \longrightarrow M(\alpha,\Delta) \longrightarrow 0. \label{type3}
\end{eqnarray}

The following are the corresponding results.
\begin{theo}\label{th2}(cf. Ref. \cite{CKW1})
Nontrivial extensions of the form \eqref{type1}\ exist if and only if  $\alpha+\gamma=0$ and $\Delta=1$ or $2$. In these cases, they are given (up to equivalence) by $$L_\lambda v_\Delta=(\partial+\alpha+\Delta\lambda)v_\Delta+f(\lambda)c_\gamma,$$ where
\begin{itemize}
\item[{\rm (i)}] $f(\lambda)=c_2\lambda^2$, for $\Delta=1$ and $c_2\neq0$.
\item [{\rm (ii)}] $f(\lambda)=c_3\lambda^3$, for $\Delta=2$ and $c_3\neq0$.
\end{itemize}
Furthermore, all trivial cocycles are given by scalar multiples of the polynomial $f(\lambda)=\alpha+\gamma+\Delta\lambda$.
\end{theo}

\begin{theo}\label{th3}(cf. Ref. \cite{CKW1})
There are nontrivial extensions of Virasoro conformal modules of the form \eqref{type2} if and only if $\alpha+\gamma=0$ and $\Delta=1$. These extensions are given (up to equivalence) by

\begin{eqnarray*}
L_\lambda c_\gamma=f(\partial,\lambda)v_\Delta,\ \
\partial c_\gamma=\gamma c_\gamma+h(\partial)v_\Delta,
\end{eqnarray*}
where $f(\partial,\lambda)=h(\partial)=a_0,$ and $a_0\neq 0$.
\end{theo}

\begin{theo}\label{th4}(cf. Ref. \cite{CKW1}) Nontrivial extensions of the form \eqref{type3} of Virasoro conformal modules exist only if $\alpha=\bar\alpha$ and $\Delta-\bar\Delta=0,1,2,3,4,5,6.$ In these cases, they are given (up to equivalence) by $$L_\lambda v_\Delta=(\partial+\alpha+\Delta\lambda)v_\Delta+f(\partial,\lambda)v_{\bar\Delta }.$$
And the following is a complete list of values of $\Delta$ and $\bar\Delta$ along with the corresponding polynomials $f(\pa,\la),$ whose nonzero scalar multiples give rise to nontrivial extensions (by replacing $\partial$ by $\partial+\alpha$):
\begin{itemize}
\item[{\rm (i)}] $\Delta=\bar\Delta, f_1(\pa,\la)=c_0+c_1\la, (c_0,c_1)\neq(0,0).$
\item [{\rm (ii)}] $\Delta=1, \bar\Delta=0, f_2(\pa,\la)=c_0\pa+c_1\pa\la+c_2\la^2,$ where $(c_0,c_1,c_2)\neq(0,0,0)$.
\item [{\rm (iii)}]$\Delta-\bar\Delta=2, f_3(\pa,\la)=\la^2(2\pa+\la)$.
\item [{\rm (iv)}]$\Delta-\bar\Delta=3, f_4(\pa,\la)=\pa\la^2(\pa+\la)$.
\item [{\rm (v)}]$\Delta-\bar\Delta=4, f_5(\pa,\la)=\la^2(4\pa^3+6\pa^2\la-\pa\la^2+\bar\Delta\la^3)$.
\item [{\rm (vi)}]$\Delta=5,\bar\Delta=0, f_6(\pa,\la)=5\pa^4\la^2+10\pa^2\la^4-\pa\la^5$.
\item [{\rm (vi')}]$\Delta=1,\bar\Delta=-4, f_{6'}(\pa,\la)=\pa^4\la^2-10\pa^2\la^4-17\pa\la^5-8\la^6$.
\item [{\rm (vii)}]$\Delta=\frac72\pm\frac{\sqrt{19}}2, \Delta-\bar\Delta=6, f_{7}(\pa,\la)=\pa^4\la^3-(2\bar\Delta+3)\pa^3\la^4-3\bar\Delta\pa^2\la^5-(3\bar\Delta+1)\pa\la^6-(\bar\Delta+\frac9{28})\la^7$.
\end{itemize}
\end{theo}
The following result, due to \cite{WY}, gives all finite free irreducible conformal modules over the Lie conformal algebras $\mathcal{W}(b)$.

\begin{prop}\label{p1} (cf. Ref. \cite{WY}) Any finite free irreducible conformal
 module over the Lie conformal algebra $\mathcal{W}(b)$
is of the form
\begin{eqnarray*}
M=\mathbb{C}[\partial]v_\Delta,\ L_\lambda
v_\Delta=(\partial+\alpha+\Delta \lambda)v_\Delta, \ H_\lambda v_\Delta=\delta_{b,0}\b v_\Delta, \ \
\end{eqnarray*}
with $\Delta,\, \alpha,\, \beta\in\C,$ and $(\Delta,\beta)\neq (0,0)$.
\end{prop}
 In this paper, we denote the $\mathcal{W}(b)$-module $M$ from Proposition \ref{p1} by $M(\alpha,\beta,\Delta)$ if $b=0$, and $M(\alpha,\Delta)$ if $b\neq0$, respectively.

\section{Extensions of conformal $\mathcal{W}(b)$-modules}
\vs{8pt}

In this section, we always assume that $b\neq0$. In this case, it follows from Proposition \ref{p1} that any finite free irreducible conformal
 module over the Lie conformal algebra $\mathcal{W}(b)$
is of the form
\begin{eqnarray*}
M(\alpha,
\Delta)=\mathbb{C}[\partial]v_\Delta,\ L_\lambda
v_\Delta=(\partial+\alpha+\Delta \lambda)v_\Delta, \ H_\lambda v_\Delta=0, \ \
\end{eqnarray*}
where $\Delta,\, \alpha\in\C,$ and $\Delta\neq 0$.


Let $M$ be a $\C[\partial]$-module. By Definition \ref{def1}, a $\mathcal{W}(b)$-module structure on $M$ is given by $L^M_\lambda, H^M_\mu \in {\rm End}_\C(M)[\lambda]$ such that
\begin{eqnarray}
&&[L^M_\lambda, L^M_\mu]=(\lambda-\mu)L^M_{\lambda+\mu},\label{WL}\\
&&[L^M_\lambda, H^M_\mu]=-(b\la+\mu)H^M_{\lambda+\mu},\label{WLH}\\
&&[H^M_\lambda, L^M_\mu]=(\lambda+b\mu)H^M_{\lambda+\mu},\label{WHL}\\
&&[H^M_\lambda, H^M_\mu]=0, \label{WH}\\
&&[\partial, L^M_\lambda]=-\lambda L^M_\lambda ,\label{Wa}\\
&&[\partial, H^M_\lambda]=-\lambda H^M_\lambda.\label{Wb}
\end{eqnarray}

First, we consider extensions of $\mathcal{W}(b)$-modules of the form
\begin{eqnarray}\label{w1m}
0\longrightarrow \C{c_\gamma}\longrightarrow E \longrightarrow M(\alpha,\Delta) \longrightarrow 0.
\end{eqnarray}
 As a module over $\C[\partial]$, $E$ in \eqref{w1m} is isomorphic to $\C {c_\gamma}\oplus M(\alpha,\Delta)$, where $\C {c_\gamma}$ is a $\mathcal{W}(b)$-submodule, and $M(\alpha,\Delta)=\C[\partial]v_\Delta$ such that the following identities hold in $E$:
\begin{eqnarray}\label{cm1}
L_\lambda v_\Delta=(\partial+\alpha+\Delta\lambda)v_\Delta+f(\lambda)c_\gamma,\ H_\lambda v_\Delta=g(\lambda)c_\gamma,
\end{eqnarray}
where $f(\lambda)=\sum_{n\geqslant0}f_n\lambda^n,\,g(\lambda)=\sum_{n\geqslant0}g_n\lambda^n,$ with $ f_n,\,g_n\in\C$.

\begin{lemm}\label{lem1} All trivial extensions of the form \eqref{w1m} are given by \eqref{cm1}, where $f(\lambda)$ is a scalar multiple of $\alpha+\gamma+\Delta\lambda$, and $g(\lambda)=0$.
\end{lemm}
\begin{proof} Suppose that \eqref{w1m} represents a trivial cocycle. This means that the exact sequence
\eqref{w1m} is split
and hence there exists $v_\Delta'=\varphi(\partial)v_\Delta+a c_\gamma\in E$, where $a\in\C$, such that
\begin{eqnarray*}
L_\lambda v_\Delta'=(\partial+\alpha+\Delta\lambda)v_\Delta'
=(\partial+\alpha+\Delta\lambda)\varphi(\partial)v_\Delta+a(\gamma+\alpha+\Delta\lambda)c_\gamma,
\end{eqnarray*}
and $H_\lambda v_\Delta'=0.$ On the other hand, it follows from \eqref{cm1} that
\begin{eqnarray*}
&&L_\lambda v_\Delta'=\varphi(\partial+\lambda)(\partial+\alpha+\Delta\lambda)v_\Delta+\varphi(\partial+\lambda)f(\lambda)c_\gamma,\\
&&H_\lambda v_\Delta'=\varphi(\partial+\lambda)g(\lambda)c_\gamma.
\end{eqnarray*}
Comparing both expressions for $L_\lambda v_\Delta'$ and $H_\lambda v_\Delta'$  respectively, we see that $\varphi(\partial)$ is a constant. Therefore $f(\lambda)$ is a scalar multiple of $\alpha+\gamma+\Delta\lambda$, and $g(\lambda)=0$.
\end{proof}
\begin{theo}\label{theo1}
There are nontrivial extensions of $\mathcal{W}(b)$-modules of the form \eqref{w1m}
if and only if $\alpha+\gamma=0$. Moreover, these nontrivial extensions are given (up to equivalence) by \eqref{cm1}, where, if $g(\lambda)=0$, then $\Delta=1,2$ and $f(\lambda)$ is from the nonzero polynomials of Theorem \ref{th2} for all $0\neq b\in\C$, or else $g(\lambda)=a$ for some $0\neq a\in\C$,  $\Delta=b$ and 
\begin{eqnarray}
f(\lambda)=\left\{
\begin{array}{ll}
c_2\lambda^2, &for\ b=1,\\
c_3\lambda^3, &for \ b=2,\\
0, &otherwise,
\end{array}
\right.
\end{eqnarray}
with $c_2,\, c_3\in\C$.


\end{theo}
\begin{proof}
Applying both sides of \eqref{WL} and \eqref{WLH} to $v_\Delta$, we obtain
\begin{eqnarray}
&&(\alpha+\gamma+\lambda+\Delta\mu)f(\lambda)-(\alpha+\gamma+\mu+\Delta\lambda)f(\mu)=(\lambda-\mu)f(\lambda+\mu),\label{wl1}\\
&&(\alpha+\gamma+\mu+\Delta\lambda)g(\mu)=(b\la+\mu) g(\lambda+\mu).\label{wlh1}
\end{eqnarray}
Setting $\lambda=0$ in \eqref{wlh1} gives
\begin{eqnarray}
(\alpha+\gamma)g(\mu)=0. \label{s4}
\end{eqnarray}
\vskip5pt
\textbf{Case 1.} $\alpha+\gamma\neq0.$
\vskip5pt

By \eqref{s4}, $g(\mu)=0$. Setting $\mu=0$ in \eqref{wl1}, we see that  $f(\lambda)$ is a scalar multiple of $\alpha+\gamma+\Delta\lambda$. Thus, the extension is trivial by Lemma \ref{lem1}.

\vskip5pt
\textbf{Case 2.} $\alpha+\gamma=0.$
\vskip5pt

Putting $\mu=0$ in \eqref{wlh1}, we obtain $g(\lambda)=\frac\Delta b g(0)$, which gives $g(\lambda)$ is a constant. If $g(\lambda)=a\neq0$, then $\Delta =b$. By Theorem \ref{th2}, $f(\lambda)=c_2\lambda^2$ if $b=1$, $f(\lambda)=c_3\lambda^3$ if $b=2$, and $f(\lambda)=0$ if $b\neq1$ or $2$. In these cases, the corresponding extensions are nontrivial.  If $g(\lambda)=0$, then it  reduces to the Virasoro case. By Theorem \ref{th2}, we obtain the result.
\end{proof}

Next we consider extensions of $\mathcal{W}(b)$}-modules of the form
\begin{eqnarray}\label{w2m}
0\longrightarrow M(\alpha,\Delta)\longrightarrow E \longrightarrow \C c_\gamma \longrightarrow 0.
\end{eqnarray}
As a vector space, $E$ in \eqref{w2m} is isomorphic to $M(\alpha,\Delta)\oplus\C {c_\gamma}$. Here $M(\alpha,\Delta)=\C[\partial]v_\Delta$ is a $\mathcal{W}(b)$}-submodule and we have
\begin{eqnarray}\label{cm2}
L_\lambda c_\gamma=f(\partial,\lambda)v_\Delta,\ \
H_\lambda c_\gamma=g(\partial,\lambda)v_\Delta,\ \
\partial c_\gamma=\gamma c_\gamma+h(\partial)v_\Delta,
\end{eqnarray}
where $f(\partial,\lambda)=\sum_{n\geqslant0}f_n(\partial)\lambda^n,\, g(\partial,\lambda)=\sum_{n\geqslant0}g_n(\partial)\lambda^n,\, f_n(\partial),\, g_n(\partial), \, h(\partial)\in\C[\partial]$.
\begin{lemm}\label{lem2} All trivial extensions of the form \eqref{w2m} are given by \eqref{cm2} with
$f(\partial,\lambda)=(\pa+\alpha+\Delta\lambda)\phi(\partial+\lambda)$, $g(\partial,\lambda)=0$ and $h(\partial)=(\partial-\gamma)\phi(\partial)$, where $\phi$ is a polynomial.
\end{lemm}
\begin{proof}  Suppose that \eqref{w2m} represents a trivial cocycle. This means that the exact sequence
\eqref{w2m} is split
and hence there exists $c'_\gamma=a c_\gamma+\phi(\partial)v_\Delta\in E$, where $a\in\C$, such that $L_\lambda c'_\gamma=H_\lambda c'_\gamma=0$ and $\partial c'_\gamma=\gamma c'_\gamma$. On the other hand, a short computation shows that
\begin{eqnarray*}
L_\lambda c'_\gamma&=&(\partial+\alpha+\Delta\lambda)\phi(\partial+\lambda)v_\Delta+a f(\partial,\lambda)v_\Delta,\\
H_\lambda c'_\gamma&=& a g(\partial,\lambda)v_\Delta,\\
\partial c'_\gamma&=& a\gamma c_\gamma+\big(  a h(\partial)+\partial\phi(\partial)\big)v_\Delta.
\end{eqnarray*}
Comparing both expressions for $L_\lambda c'_\gamma$, $H_\lambda c'_\gamma$ and $\partial c'_\gamma$ respectively, we obtain the result.
\end{proof}

\begin{theo}\label{theo2}
There are nontrivial extensions of $\mathcal{W}(b)$-modules of the form \eqref{w2m} if and only if $\alpha+\gamma=0$ and $\Delta=1$. In this case, ${\rm dim}_\C\big(\C{c_{-\alpha}},{{\rm Ext}} (M(\alpha,1))\big)=1,$  and the unique (up to a scalar) nontrivial extension is given by
\begin{eqnarray*}
L_\lambda c_\gamma=a v_\Delta,\ \
H_\lambda c_\gamma=0,\ \
\partial c_\gamma=\gamma c_\gamma+a v_\Delta,
\end{eqnarray*}
where $a$ is a nonzero complex number.
\end{theo}

\begin{proof}

Applying both sides of \eqref{WL}, \eqref{Wa} and \eqref{Wb} to $c_\gamma$ gives the following functional equations:
\begin{eqnarray}
&&(\partial +\alpha +\Delta \lambda )f(\partial +\lambda ,\mu )-(\partial +\alpha +\Delta \mu )f(\alpha +\mu ,\lambda )=(\lambda -\mu )f(\partial ,\lambda +\mu ),\label{s5}\\
&&(\partial +\lambda-\gamma)f(\partial ,\lambda )=(\partial +\alpha +\Delta \lambda )h(\partial +\lambda ),\label{s6}\\
&&(\partial +\lambda-\gamma)g(\partial ,\lambda )=0.\label{s7}
\end{eqnarray}
Obviously, $g(\partial ,\lambda )=0$ by \eqref{s7}. Replacing $\partial$ by $\bar\partial=\partial+\alpha$ and letting $\bar f(\bar\partial, \lambda) = f(\bar\partial-\alpha, \lambda)$, and $\bar h(\bar\partial) = h(\bar\partial-\alpha)$, we can rewrite
\eqref{s5} and \eqref{s6} in a homogeneous form
\begin{eqnarray}
&&(\bar\partial+\Delta\lambda)\bar f(\bar\partial+\lambda ,\mu )-(\bar\partial +\Delta \mu )\bar f(\bar\partial+\mu ,\lambda )=(\lambda -\mu )\bar f(\bar \partial ,\lambda +\mu ),\label{s9}\\
&&(\bar\partial-\alpha +\lambda-\gamma)\bar f(\bar \partial ,\lambda )=(\bar\partial +\Delta \lambda )\bar h(\bar\partial +\lambda ).\label{s10}
\end{eqnarray}
Taking $\mu=0$ in \eqref{s9}, we can obtain that, if degree of $\bar f$ is positive, $\bar f(\bar\partial, \lambda)$ is a scalar multiple of $(\bar\partial+\Delta\lambda)\bar f(\bar \partial+\lambda)$, where $\bar f(\bar\partial+\lambda)$ is a polynomial in $\bar \partial+\lambda$, or else $\bar f(\bar\partial,\lambda)=a\in\C.$

Assume that $\bar f(\bar\partial, \lambda)$ is a scalar multiple of $(\bar\partial+\Delta\lambda)\bar f(\bar \partial+\lambda)$. Letting $\lambda=0$ in  \eqref{s10}, we have $\bar h(\bar\partial)$ is a scalar multiple of $(\bar\partial-\alpha-\gamma)\bar f(\bar\partial)$. However, they correspond to the trivial extension by Lemma \ref{lem2}.

It is left to consider the case $\bar f(\bar \partial,\lambda)=a$. Substituting this into \eqref{s10} gives $\bar h(\bar \partial)=\bar f(\bar \partial,\lambda)=a$, and in particular, $\alpha+\gamma=0,$ $\Delta=1$ if $a\neq 0$. The proof is finished.
\end{proof}

Finally, we study extensions of conformal $\mathcal{W}(b)$-modules of the form
\begin{eqnarray}\label{3m}
0\longrightarrow M(\bar\alpha,\bar\Delta)\longrightarrow E \longrightarrow M(\alpha,\Delta) \longrightarrow 0,
\end{eqnarray}
where $\Delta, \bar\Delta \neq 0$. As a $\C[\partial]$-module, $E\cong\C[\partial]v_{\bar\Delta}\oplus\C[\partial]v_\Delta$. The following identities hold in $E$
\begin{eqnarray}\label{3m*}
L_\lambda v_\Delta=(\partial+\alpha+\Delta\lambda)v_\Delta+f(\partial,\lambda)v_{\bar\Delta },\ \
H_\lambda v_\Delta=g(\partial ,\lambda )v_{\bar\Delta},
\end{eqnarray}
where $f(\partial,\lambda)=\sum_{n\geqslant0}f_n(\partial)\lambda^n,\ g(\partial,\lambda)=\sum_{n\geqslant0}g_n(\partial)\lambda^n, f_n(\partial),g_n(\partial)\in\C[\partial]$ and $f_n(\partial)=g_n(\partial)=0$ for $n\gg0$.

\begin{lemm}\label{lem4} All trivial extensions of the form \eqref{3m} are given by \eqref{3m*}, where
$f(\partial,\lambda)$ is a scalar multiple of $(\partial+\alpha+\Delta\lambda)\phi(\partial)-(\partial+\bar\alpha+\bar\Delta\lambda)\phi(\partial+\lambda)$ and $g(\partial,\lambda)=0$, where $\phi$ is a polynomial.
\end{lemm}
\begin{proof} Suppose that \eqref{3m} represents a trivial cocycle. This means that the exact sequence
\eqref{3m} is split
and hence there exists $v_\Delta'=\psi (\partial)v_\Delta+\phi(\partial)v_{\bar\Delta}\in E$, such that
\begin{eqnarray*}
L_\lambda v_\Delta'=(\partial+\alpha+\Delta\lambda)v_\Delta'
=(\partial+\alpha+\Delta\lambda)\big(\psi(\partial)v_\Delta+\phi(\partial)v_{\bar\Delta}\big),
\end{eqnarray*}
and $H_\lambda v'_\Delta=0.$ On the other hand, a short computation shows that
\begin{eqnarray*}
L_\lambda v_\Delta'&=&L_\lambda (\psi (\partial)v_\Delta+\phi(\partial)v_{\bar\Delta})\\
&=&\psi(\partial+\lambda)L_\lambda v_\Delta+\phi(\partial+\lambda)L_\lambda v_{\bar\Delta}\\
&=&(\partial+\alpha+\Delta\lambda)\psi(\partial+\lambda)v_\Delta+\big(\psi(\partial+\lambda)f(\partial,\lambda)+\phi(\partial+\lambda)(\partial+\bar\alpha+\bar\Delta\lambda)\big)v_{\bar\Delta}
\end{eqnarray*}
Comparing both expressions for $L_\lambda v_\Delta'$ gives
\begin{eqnarray*}
&&(\partial+\alpha+\Delta\lambda)\psi(\partial)=(\partial+\alpha+\Delta\lambda)\psi(\partial+\lambda),\\
&&(\partial+\alpha+\Delta\lambda)
\phi(\partial)=f(\partial,\lambda)\psi(\partial+\lambda)+(\partial+\bar\alpha+\bar\Delta\lambda)\phi(\partial+\lambda),
\end{eqnarray*}
which imply that $\psi(\partial)=a$ with $a$ being a nonzero complex number  and hence $f(\partial,\lambda)$ is a scalar multiple of $(\partial+\alpha+\Delta\lambda)\phi(\partial)-(\partial+\bar\alpha+\bar\Delta\lambda)\phi(\partial+\lambda)$.
Similarly, we have
\begin{eqnarray*}
H_\lambda v_\Delta'=a g(\partial,\lambda)v_{\bar\Delta}=0,
\end{eqnarray*}
leading to $g(\partial,\lambda)=0$.
\end{proof}

Applying both sides of \eqref{WL} and \eqref{WLH} to $v_\Delta$ gives the following functional equations:

\begin{eqnarray}
(\lambda-\mu)f(\partial,\lambda+\mu)&=&(\partial+\lambda+\Delta\mu+\alpha)f(\partial,\lambda)+(\partial+\bar\Delta\lambda+\bar\alpha)f(\partial+\lambda,\mu)\nonumber \\
&&-(\partial+\mu+\Delta\lambda+\alpha)f(\partial,\mu)-(\partial+\bar\Delta\mu+\bar\alpha)f(\partial+\mu,\lambda),\label{wl}\\
-(b\lambda+\mu)g(\partial,\lambda+\mu)&=&(\partial+\bar\Delta\lambda+\bar\alpha)g(\partial+\lambda,\mu)-(\partial+\mu+\Delta\lambda+\alpha)g(\partial,\mu).\label{wlh}
\end{eqnarray}
Putting $\lambda=0$ in \eqref{wl}and \eqref{wlh} gives that
\begin{eqnarray}\label{14}
(\alpha-\bar\alpha)f(\partial,\mu)&=&(\partial+\Delta\mu+\alpha)f(\partial,0)-(\partial+\bar\Delta\mu+\bar\alpha)f(\partial+\mu,0),\label{3l1}\\
(\alpha-\bar\alpha)g(\partial,\mu)&=&0.\label{3lh1}
\end{eqnarray}

\textbf{Case 1.} $\alpha\neq\bar\alpha$.
\vskip5pt

By \eqref{3l1} and \eqref{3lh1}, we have $f(\partial,\mu)=\frac1{\alpha-\bar\alpha}\big((\partial+\Delta\mu+\alpha)f(\partial,0)-(\partial+\bar\Delta\mu+\bar\alpha)f(\partial+\mu,0)\big)$ and $g(\partial,\mu)=0$. This corresponds to the trivial extension by Lemma \ref{lem4}.

\vskip5pt
\textbf{Case 2.} $\alpha=\bar\alpha$.
\vskip5pt

For convenience, we put $\bar\partial=\partial+\alpha$ and let $\bar f(\bar\partial,\lambda)=f(\bar\partial-\alpha,\lambda)$, $\bar g(\bar\partial,\lambda)=g(\bar\partial-\alpha,\lambda)$ .
 In what follows we will continue to write $\partial$ for $\bar\partial$, $f$ for $\bar f$ and $g$ for $\bar g$.  Now we can rewrite \eqref{wl} and \eqref{wlh} as follows:
\begin{eqnarray}
(\lambda-\mu)f(\partial,\lambda+\mu)&=&(\partial+\lambda+\Delta\mu)f(\partial,\lambda)+(\partial+\bar\Delta\lambda)f(\partial+\lambda,\mu)\nonumber \\
&&-(\partial+\mu+\Delta\lambda)f(\partial,\mu)-(\partial+\bar\Delta\mu)f(\partial+\mu,\lambda),\label{svl1}\\
-(b\lambda+\mu)g(\partial,\lambda+\mu)&=&(\partial+\bar\Delta\lambda)g(\partial+\lambda,\mu)-(\partial+\mu+\Delta\lambda)g(\partial,\mu).\label{svlm1}
\end{eqnarray}
Setting $\mu=0$ in \eqref{svlm1} gives
\begin{eqnarray}\label{yy}
b\lambda g(\partial,\lambda)=(\partial+\Delta\lambda)g(\partial,0)-(\partial+\bar\Delta\lambda)g(\partial+\lambda,0).
\end{eqnarray}
By the nature of \eqref{svlm1}, we may assume that a solution is a homogeneous polynomial in $\partial$ and $\lambda$ of degree $m$. Hence we will assume from now on that $g(\partial,\lambda)=\sum^{m}_{i=0}a_i\partial^{m-i}\lambda^i$, $a_i\in\C$. Substituting this into \eqref{yy} gives
\begin{eqnarray}\label{coeff}
\sum^{m}_{i=0}ba_i\partial^{m-i}\lambda^{i+1}=(\partial+\Delta\lambda)a_0\partial^m-(\partial+\bar\Delta\lambda)a_0(\partial+\lambda)^m.
\end{eqnarray}
We see that $a_i=0$ for $i=0, 1,\cdots,m$ if $a_0=0$. Then $g(\partial,\lambda)=0$, and it reduces to  the Virasoro case.

Now we assume that $a_0\neq0$. Comparing the coefficients of $\partial^{m}\lambda$ in \eqref{coeff} gives
\begin{eqnarray}\label{no.1}
\Delta-\bar\Delta=m+b.
\end{eqnarray}
Equating the coefficients of $\partial^{m-i}\lambda^{i+1}$ in both sides of \eqref{coeff}, we have
\begin{eqnarray}\label{coe}
ba_i=-a_0\binom{m}{i+1}-a_0\bar\Delta\binom{m}{i},\ \ 1\leqslant i \leqslant m.
\end{eqnarray}
Plugging $g(\partial,\lambda)=\sum^{m}_{i=0}a_i\partial^{m-i}\lambda^i$ into \eqref{svlm1} gives
\begin{eqnarray}\label{coeff1}
-(b\lambda+\mu)\sum^{m}_{i=0}a_i\partial^{m-i}(\lambda+\mu)^i=(\partial+\bar\Delta\lambda)\sum^{m}_{i=0}a_i(\partial+\lambda)^{m-i}\mu^i-(\partial+\mu+\Delta\lambda)\sum^{m}_{i=0}a_i\partial^{m-i}\mu^i.
\end{eqnarray}
Comparing the coefficients of $\lambda^m\mu$ in \eqref{coeff1} with $m\geqslant 2$ gives
$(-bm-1)a_m=\bar\Delta a_1$ and hence by \eqref{coe},
\begin{eqnarray}\label{coee1}
(-bm-1)\bar\Delta=\big(\mbox{$\binom{m}{2}$}+m\bar\Delta\big)\bar\Delta.
\end{eqnarray}
Since $\bar\Delta\neq0$,
\begin{eqnarray}\label{coe1}
-bm-1=\mbox{$\binom{m}{2}$}+m\bar\Delta.
\end{eqnarray}
Similarly, collecting the coefficients of $\partial\lambda^{m-1}\mu$ in \eqref{coeff1} with $m\geqslant3$ and using \eqref{coe} again, we have
\begin{eqnarray}\label{coe2}
\big(-b (m-1)-1\Big)(1+m\bar\Delta)=\big(1+(m-1)\bar\Delta\big)\Big(\mbox{$\binom{m}{2}$}+m\bar\Delta\Big).
\end{eqnarray}
Combining \eqref{coe1} with \eqref{coe2} gives $\bar\Delta=b.$  Plugging this into \eqref{coe1}, we obtain
\begin{eqnarray}\label{coe1-1}
\mbox{$\binom{m}{2}$}+2bm+1=0.
\end{eqnarray}
Assume that $m\geqslant4$. In this case the coefficients of $\partial^{m-2}\lambda^2\mu$ in \eqref{coeff1} gives [using $\bar\Delta=b$ and \eqref{coe} again]
\begin{eqnarray}\label{coe4}
(-2b-1)\Big(\mbox{$\binom{m}{3}$}+\mbox{$\binom{m}{2}$}b\Big)=\Big(\mbox{$\binom{m-1}{2}$}+(m-1)b\Big)\Big(\mbox{$\binom{m}{2}$}+bm\Big).
\end{eqnarray}
Combining \eqref{coe1-1} with \eqref{coe4}, we have $m=-2, -1, 2$ or $3$, a contradiction.
Therefore, $m$ is at most $3$, i.e., $ m=0,1,2,3$.

For $m=0$, we have $\Delta-\bar\Delta=b$ and $g(\partial,\lambda)=a_0 $ is a solution to \eqref{svlm1}.

For $m=1$, we have $\Delta-\bar\Delta=1+b$. And it follows from \eqref{coe} that $g(\partial,\lambda)=a_0(\partial-\frac1b\bar\Delta\lambda)$, which is easily checked to be a solution to
\eqref{svlm1}.

Now assume that $m=2$. By \eqref{no.1} and \eqref{coe1}, we have  $\bar\Delta=-1-b$ and $\Delta=1$. By
\eqref{coe}, we may assume that    $g(\partial,\lambda)=\partial^2-(\frac1b+\frac2b\bar\Delta)\pa\la-\frac1b\bar\Delta\lambda^2$.  Plugging this back into \eqref{svlm1}, we obtain
\begin{eqnarray*}
&&-(b\lambda+\mu)\big(\partial^2-(\frac1b+\frac2b\bar\Delta)\pa(\la+\mu)-\frac1b\bar\Delta(\lambda+\mu)^2\big)\\
&&\ \ \ \ \ \ \ \ \ \, =(\partial+\bar\Delta\lambda)\big((\partial+\lambda)^2-(\frac1b+\frac2b\bar\Delta)(\pa+\la)\mu-\frac1b\bar\Delta\mu^2\big)\\
&&\ \ \ \ \ \ \ \ \ \ \ \ \ \ \, -(\partial+\mu+\Delta\lambda)\big(\partial^2-(\frac1b+\frac2b\bar\Delta)\pa\mu-\frac1b\bar\Delta\mu^2\big),
\end{eqnarray*}
which holds provided that $\bar\Delta=-1-b$ and $\Delta=1$.

Finally we consider the case $m=3$.  By \eqref{coe1-1} and \eqref{no.1}, we have $\bar\Delta=b=-\frac23$ and $\Delta=\frac53.$ By
\eqref{coe}, we may assume that  $g(\partial,\lambda)=\partial^3+\frac32\pa^2\la-\frac32\pa\lambda^2-\lambda^3$.   Plugging this back into \eqref{svlm1}, we obtain that it is a solution.

The above discussions prove the following:
\begin{lemm}\label{lemm4-4}
Let $g(\partial,\lambda)$ be a nonzero homogeneous polynomial of degree $m$ satisfying \eqref{svlm1}.
Then $\Delta-\bar\Delta=m+b$ and $m\leqslant 3$. Furthermore, we have
\begin{itemize}
\item[{\rm (1)}]
For $b=-\frac 23$, all the solutions (up to a scalar) to \eqref{svlm1} are given by
\begin{itemize}
\item[{\rm (i)}] $m=0$, $\Delta-\bar\Delta=-\frac23$, and $g(\partial,\lambda)=1$;
\item [{\rm (ii)}] $m=1$, $\Delta-\bar\Delta=\frac13 $, and $g(\partial,\lambda)=\partial+\frac32\bar\Delta\lambda$;
\item [{\rm (iii)}] $m=2$, $\Delta=1$, $\bar\Delta=-\frac13$, and $g(\partial,\lambda)=\partial^2+\frac12\pa\la-\frac12\lambda^2$;
\item [{\rm (iv)}] $m=3$, $\Delta=\frac53$, $\bar\Delta=-\frac23$, and $g(\partial,\lambda)=\partial^3+\frac32\pa^2\la-\frac32\pa\lambda^2-\lambda^3$,
\end{itemize}

\item[{\rm (2)}] For $b\neq-\frac 23$, all the solutions (up to a scalar) to \eqref{svlm1} are given by
\begin{itemize}
\item[{\rm (i)}] $m=0$, $\Delta-\bar\Delta=b$, and $g(\partial,\lambda)=1$;
\item [{\rm (ii)}]$m=1$, $\Delta-\bar\Delta=1+b $, and $g(\partial,\lambda)=\partial-\frac1b\bar\Delta\lambda;$
\item [{\rm (iii)}]$m=2$, $\Delta=1$, $\bar\Delta=-1-b$ and $g(\partial,\lambda)=\partial^2-(\frac1b+\frac2b\bar\Delta)\pa\la-\frac1b\bar\Delta\lambda^2$.
\end{itemize}
\end{itemize}
\end{lemm}

By Lemma \ref{lemm4-4} and Theorem \ref{th4}, we obtain the following.
\begin{theo}\label{theo3}
Nontrivial extensions of $\mathcal{W}(b)$-modules of the form \eqref{3m} exist only if $\alpha=\bar\alpha$. For each $\alpha\in\C,$ these extensions are given, up to equivalence, by
\begin{eqnarray*}
L_\lambda v_\Delta=(\partial+\alpha+\Delta\lambda)v_\Delta+f(\partial,\lambda)v_{\bar\Delta },\ \
H_\lambda v_\Delta=g(\partial ,\lambda)v_{\bar\Delta},
\end{eqnarray*}
where $g(\partial ,\lambda)=0$ and $f(\partial,\lambda)$ is from the nonzero polynomials of Theorem \ref{th4} with $\Delta,\bar\Delta \neq 0$, or the values of $\Delta$ and $\bar\Delta$ along with the
pairs of polynomials $g(\partial,\lambda)$ and $f(\partial,\lambda)$, whose nonzero scalar multiples give rise to nontrivial extensions,
are listed as follows  (by replacing $\partial$ by $\partial+\alpha$):
\begin{itemize}
\item[{\rm (1)}] When $b=-1$, we have $\Delta-\bar\Delta=-1$ or $0$. And
\begin{itemize}
\item[{\rm (i)}] In the case $\Delta-\bar\Delta=-1$, $f(\pa,\la)=0$ and $g(\pa,\la)=a_0$, where $a_0\neq0$.
\item[{\rm (ii)}] In the case $\Delta-\bar\Delta=0$, $f(\pa,\la)=c_0+c_1\la$, and $g(\pa,\la)=\partial+\bar\Delta\lambda$, where $c_0, c_1\in\C.$
\end{itemize}

\item[{\rm (2)}] When  $b=1$, we have $\Delta-\bar\Delta=1$, $2$ or $\Delta=1,\bar\Delta=-2$. And
\begin{itemize}
\item[{\rm (i)}] In the case $\Delta-\bar\Delta=1$, $f(\pa,\la)=0$ and $g(\pa,\la)=a_0$, where $a_0\neq0$.
\item[{\rm (ii)}] In the case $\Delta-\bar\Delta=2$, $f(\pa,\la)=c_0\la^2(2\pa+\la)$ and $g(\pa,\la)=\partial-\bar\Delta\lambda$, where $c_0\in\C$.
\item[{\rm (iii)}] In the case $\Delta=1, \bar\Delta=-2$, $f(\pa,\la)=c_0\pa\la^2(\pa+\la)$ and $g(\pa,\la)=\partial^2+3\pa\la+2\lambda^2$, where $c_0\in\C$.
\end{itemize}
\item[{\rm (3)}] When  $b=2$, we have $\Delta-\bar\Delta=2$, $3$ or $\Delta=1,\bar\Delta=-3$. And
\begin{itemize}
\item[{\rm (i)}] In the case $\Delta-\bar\Delta=2$, $f(\pa,\la)=c_0\la^2(2\pa+\la)$ and $g(\pa,\la)=a_0$, where $c_0\in\C, a_0\neq0$.
\item[{\rm (ii)}] In the case $\Delta-\bar\Delta=3$, $f(\pa,\la)=c_0\pa\la^2(\pa+\la)$ and $g(\pa,\la)=\partial-\frac12\bar\Delta\lambda$, where $c_0\in\C$.
\item[{\rm (iii)}] In the case $\Delta=1,\bar\Delta=-3$, $f(\pa,\la)=c_0\la^2(4\pa^3+6\pa^2\la-\pa\la^2+3\la^3)$ and $g(\pa,\la)=\partial^2+\frac52\pa\la+\frac32\lambda^2$, where $c_0\in\C$.
\end{itemize}
\item[{\rm (4)}] When  $b=3$, we have $\Delta-\bar\Delta=3$, $4$ or $\Delta=1,\bar\Delta=-4$. And
\begin{itemize}
\item[{\rm (i)}] In the case $\Delta-\bar\Delta=3$, $f(\pa,\la)=c_0\pa\la^2(\pa+\la)$ and $g(\pa,\la)=a_0$, where $c_0\in\C, a_0\neq0$.
\item[{\rm (ii)}] In the case $\Delta-\bar\Delta=4$, $f(\pa,\la)=c_0\la^2(4\pa^3+6\pa^2\la-\pa\la^2+\bar\Delta\la^3)$ and $g(\pa,\la)=\partial-\frac13\bar\Delta\lambda$, where $c_0\in\C$.
\item[{\rm (iii)}] In the case  $\Delta=1,\bar\Delta=-4$, $f(\pa,\la)=c_0(\pa^4\la^2-10\pa^2\la^4-17\pa\la^5-8\la^6)$ and $g(\pa,\la)=\partial^2+\frac73\pa\la+\frac43\lambda^2$, where $c_0\in\C$.
\end{itemize}
\item[{\rm (5)}] When $b=4$, we have $\Delta-\bar\Delta=4$, $5$ or $\Delta=1,\bar\Delta=-5$. And
\begin{itemize}
\item[{\rm (i)}] In the case $\Delta-\bar\Delta=4$, $f(\pa,\la)=c_0\la^2(4\pa^3+6\pa^2\la-\pa\la^2+\bar\Delta\la^3)$ and $g(\pa,\la)=a_0$, where $c_0\in\C, a_0\neq0$.
\item[{\rm (ii)}] In the case $\Delta-\bar\Delta=5$, $f(\pa,\la)=0$ and $g(\pa,\la)=\partial-\frac14\bar\Delta\lambda$.
\item[{\rm (ii')}] In the case $\Delta=1,\bar\Delta=-4$, $f(\pa,\la)=c_0(\pa^4\la^2-10\pa^2\la^4-17\pa\la^5-8\la^6)$
 and $g(\pa,\la)=\partial+\lambda$, where $c_0\in\C$.

\item[{\rm (iii)}] In the case $\Delta=1,\bar\Delta=-5$, $f(\pa,\la)=0$ and $g(\pa,\la)=\partial^2+\frac{9}{4}\pa\la+\frac54\lambda^2$.
\end{itemize}

\item[{\rm (6)}] When  $b=5$, we have $\Delta-\bar\Delta=5$, $6$ or $\Delta=1,\bar\Delta=-6$. And
\begin{itemize}
\item[{\rm (i)}] In the case $\Delta-\bar\Delta=5$, $f(\pa,\la)=0$ and $g(\pa,\la)=a_0$, where $a_0\neq0$.
\item[{\rm (i')}] In the case $\Delta=1,\bar\Delta=-4$, $f(\pa,\la)=c_0(\pa^4\la^2-10\pa^2\la^4-17\pa\la^5-8\la^6)$ and $g(\pa,\la)=a_0$, where $c_0\in\C,$ $a_0\neq0$.

\item[{\rm (ii)}] In the case $\Delta-\bar\Delta=6$, $f(\pa,\la)=0$ and $g(\pa,\la)=\partial-\frac15\bar\Delta\lambda$.
\item[{\rm (ii')}] In the case $\Delta=\frac72\pm\frac{\sqrt{19}}2,\Delta-\bar\Delta=6, f(\pa,\la)=c_0\big(\pa^4\la^3-(2\bar\Delta+3)\pa^3\la^4-3\bar\Delta\pa^2\la^5-(3\bar\Delta+1)\pa\la^6-(\bar\Delta+\frac9{28})\la^7\big)$ and $g(\pa,\la)=\partial-\frac15\bar\Delta\lambda$, where $c_0\in\C$.

 \item[{\rm (iii)}] In the case $\Delta=1,\bar\Delta=-6$, $f(\pa,\la)=0$ and $g(\pa,\la)=\partial^2+\frac{11}5\pa\la+\frac65\lambda^2$.
 \end{itemize}
\item[{\rm (7)}] When  $b=6$, we have $\Delta-\bar\Delta=6$, $7$ or $\Delta=1, \bar\Delta=-7$. And
\begin{itemize}
\item[{\rm (i)}] In the case $\Delta-\bar\Delta=6$, $f(\pa,\la)=0$ and $g(\pa,\la)=a_0$, where $a_0\neq0$.
\item[{\rm (i')}] In the case $\Delta=\frac72\pm\frac{\sqrt{19}}2, \bar\Delta=-\frac52\pm\frac{\sqrt{19}}2, f(\pa,\la)=c_0\big(\pa^4\la^3-(2\bar\Delta+3)\pa^3\la^4-3\bar\Delta\pa^2\la^5-(3\bar\Delta+1)\pa\la^6-(\bar\Delta+\frac9{28})\la^7\big)$ and $g(\pa,\la)=a_0$, where $c_0\in\C,$ $a_0\neq0$.

\item[{\rm (ii)}]In the case $\Delta-\bar\Delta=7$, $f(\pa,\la)=0$ and $g(\pa,\la)=\partial-\frac16\bar\Delta\lambda$.
\item[{\rm (iii)}]In the case $\Delta=1,\bar\Delta=-7$, $f(\pa,\la)=0$ and $g(\pa,\la)=\partial^2+\frac{13}{6}\pa\la+\frac76\lambda^2$.
\end{itemize}
\item[{\rm (8)}] When  $b=-\frac23$, we have $f(\pa,\la)=0$ and the values $\Delta$ and $\bar\Delta$ along with $g(\partial,\lambda)$ are from Lemma \ref{lemm4-4}(1).
\item[{\rm (9)}] When  $b\neq-1, 1, 2, 3, 4, 5, 6$ or $-\frac23$,  $f(\pa,\la)=0$ and the values $\Delta$ and $\bar\Delta$ along with $g(\partial,\lambda)$ are from Lemma \ref{lemm4-4}(2).
\end{itemize}
\end{theo}

\vs{8pt}

\noindent\bf{ Acknowledgements.}\ \rm
{\footnotesize This work was supported by National Natural Science
Foundation of China (11301109) and China Scholarship Council.

\end{document}